\documentclass[12pt]{amsart}

\usepackage{amsmath,amsthm,amssymb}
\usepackage[top=30truemm,bottom=30truemm,left=30truemm,right=30truemm]{geometry}
\usepackage[dvipdfmx,linktocpage]{hyperref}
\usepackage[all]{xy}
\usepackage[dvipsnames]{xcolor}
\usepackage{threeparttable}
\usepackage{yhmath}

\definecolor{mycolor}{rgb}{0.89,0,0.87}
\hypersetup{
    colorlinks=true,
    citecolor=mycolor,
    linkcolor=mycolor,
}

\theoremstyle{plain}
\newtheorem{theorem}{Theorem}[section]
\newtheorem{proposition}{Proposition}[section]
\newtheorem{lemma}{Lemma}[section]

\theoremstyle{definition}
\newtheorem{definition}{Definition}[section]

\theoremstyle{remark}
\newtheorem{remark}{Remark}[section]
\newtheorem{example}{Example}[section]

\newcommand{\End}{\hfill{\blacksquare}}

\newcommand{\prarrow}[2]{\ar@<0.5ex>[r]^-{#1} \ar@<-0.5ex>[r]_-{#2}}
\newcommand{\plarrow}[2]{\ar@<0.5ex>[l]^-{#1} \ar@<-0.5ex>[l]_-{#2}} 
\newcommand{\pdarrow}[2]{\ar@<0.5ex>[d]^-{#1} \ar@<-0.5ex>[d]_-{#2}} 
\newcommand{\puarrow}[2]{\ar@<0.5ex>[u]^-{#1} \ar@<-0.5ex>[u]_-{#2}}

\begin{document}

\title{Reduction of coK\"{a}hler and 3-cosymplectic manifolds}
\author{Shuhei Yonehara}
\address{Department of Mathematics, Graduate School of Science, Osaka University, Toyonaka, Osaka 560-0043, Japan}
\email{u074738h@ecs.osaka-u.ac.jp}

\begin{abstract}
    The notions of coK\"{a}hler  manifolds and 3-cosymplectic manifolds are odd-dimensional analogues of the ones of K\"{a}hler manifolds and hyperK\"{a}hler manifolds, respectively. In this paper, we obtain reduction theorems of coK\"{a}hler manifolds and 3-cosymplectic manifolds. We show that K\"{a}hler and coK\"{a}hler (hyperK\"{a}hler and 3-cosymplectic) reductions admit a natural compatibility with respect to ``cylinder constructions''. We further prove the compatibility of K\"{a}hler and coK\"{a}hler reductions with respect to the ``mapping torus construction''.
 \end{abstract}
  
\maketitle

\tableofcontents


\section{Introduction}

Since the pioneering work of Marsden-Weinstein \cite{marsden1974reduction}, many types of reduction theorems have been studied for various geometric structures on manifolds (see the table below). 
\vspace{0.1in}
\begin{center}
    \begin{threeparttable}[htbp]
        \begin{tabular}{|c||c|c|}
            \hline
            Geometric structures&Structures with metric&3-structures with metric\\ 
          \hline
          \hline
          Symplectic\ \cite{marsden1974reduction}&K\"{a}hler\ \cite{hitchin1987hyperkahler}&HyperK\"{a}hler\ \cite{hitchin1987hyperkahler}\\ 
          \hline
          Contact \cite{albert1989theoreme}&Sasakian \cite{grantcharov2001reduction}&3-Sasakian \cite{boyer1994geometry}\\ 
          \hline
          Cosymplectic\ \cite{albert1989theoreme}&CoK\"{a}hler&3-cosymplectic\\ 
          \hline
        \end{tabular}
    \end{threeparttable}
\end{center}

\vspace{0.1in}
In \cite{hitchin1987hyperkahler}, Hitchin {\it et al} proved the reduction theorem of K\"{a}hler manifolds. They also introduced the notion of hyperK\"{a}hler moment maps and proved the reduction theorem of hyperK\"{a}hler manifolds. Albert \cite{albert1989theoreme} studied Hamiltonian actions on contact manifolds and cosymplectic manifolds and proved the reduction theorems. Afterwards, several types of reduction theorems of contact manifolds have been studied \cite{geiges1997constructions,willett2002contact,zambon2006contact}. In \cite{boyer1994geometry}, Boyer {\it et al} proved the reduction theorem of 3-Sasakian manifolds via the hyperK\"{a}hler reduction theorem. Afterwards, Grantcharov and Ornea \cite{grantcharov2001reduction} proved the reduction theorem of Sasakian manifolds.

In this paper, we focus on the reduction theorem of cosymplectic manifolds (see \cite{cappelletti2013survey} and \cite{yoldas2021some} for more details about cosymplectic manifolds) proved by Albert and obtain reduction theorems of \textit{coK\"{a}hler manifolds} and \textit{3-cosymplectic manifolds}. They are another odd-dimensional versions of K\"{a}hler and hyperK\"{a}hler manifolds instead of Sasakian and 3-Sasakian manifolds, respectively. Typical examples of coK\"{a}hler/3-cosymplectic quotients are given by K\"{a}hler/hyperK\"{a}hler quotients by using ``cone constructions'' and ``mapping torus constructions'', respectively.

This paper is organized as follows. In \autoref{cosymplectic}, we recall cosymplectic structures and cosymplectic moment maps and the proof of the reduction theorem by Albert. In \autoref{cokahler} we prove the following coK\"{a}hler reduction theorem, which is a natural analogue of the K\"{a}hler reduction theorem.
\begin{theorem}[\autoref{100}]
    Let $(M,g,\varphi,\xi,\eta)$ be a coK\"{a}hler manifold with the underlying cosymplectic structure $(\eta,\omega)$. Suppose that there is a free and proper Hamiltonian action of a Lie group $G$ on $(M,\eta,\omega)$ which preserves $\varphi$. Let $\mu:M\to\mathfrak{g}^\ast$ be a moment map and $\zeta\in\mathfrak{g}^\ast$ a central and regular value of $\mu$.
     Then $M^\zeta:=\mu^{-1}(\zeta)/G$ admits a coK\"{a}hler structure $(g^\zeta,\varphi^\zeta,\xi^\zeta,\eta^\zeta)$. Moreover, the underlying cosymplectic manifold of $(M^\zeta,g^\zeta,\varphi^\zeta,\xi^\zeta,\eta^\zeta)$ is the cosymplectic quotient $(M^\zeta,\eta^\zeta,\omega^\zeta)$. \hfill{$\Box$}
\end{theorem}

In \autoref{3-cosymplectic} we introduce a notion of 3-cosymplectic moment maps and prove the following 3-cosymplectic reduction theorem, which is a natural analogue of the hyperK\"{a}hler reduction theorem. 
\begin{theorem}[\autoref{200}]
    Let $(M,g,(\varphi_i,\xi_i,\eta_i)_{i=1,2,3})$ be a 3-cosymplectic manifold with underlying cosymplectic structures $(\eta_i,\omega_i)_{i=1,2,3}$. Suppose that there is a free and proper action of a Lie group $G$ on $M$ which is Hamiltonian with respect to all three cosymplectic structures $(\eta_i,\omega_i)_{i=1,2,3}$ and preserves $(\varphi_i)_{i=1,2,3}$. Let $\mu:M\to\mathfrak{g}^\ast\otimes{\rm Im}\mathbb{H}$ be a 3-cosymplectic moment map and $\zeta\in\mathfrak{g}^\ast\otimes{\rm Im}\mathbb{H}$ a central and regular value of $\mu$. Then $M^\zeta:=\mu^{-1}(\zeta)/G$ inherits the 3-cosymplectic structure of $M$. \hfill{$\Box$}
\end{theorem}

In \autoref{cone}, we study the reduction of geometric structures on cones. Let $M$ be a K\"{a}hler manifold. Then its \textit{cone} $C(M):=M\times \mathbb{R}$ admits a coK\"{a}hler structure. Conversely, if $M$ is a coK\"{a}hler manifold, then $C(M)$ admits a K\"{a}hler structure. We show that the K\"{a}hler (resp. coK\"{a}hler) quotient of $C(M)$ is the cone of the coK\"{a}hler (resp. K\"{a}hler) quotient of $M$. Similarly, hyperK\"{a}hler structures and 3-cosymplectic structures are also related by cone constructions, and we also show that hyperK\"{a}hler/3-cosymplectic reduction procedures are compatible with these cone constructions.

In \autoref{mapping}, we investigate coK\"{a}hler quotients of mapping tori of K\"{a}hler manifolds. For a K\"{a}hler manifold $S$ and a Hermitian isometry $f$ of $S$, the mapping torus
\[S_f=(S\times[0,1])/\{(p,0)\sim(f(p),1)\mid p\in S\}\]
admits a coK\"{a}hler structure. Suppose that there is a free and proper Hamiltonian action of a Lie group on $S$ which preserves the K\"{a}hler structure and let $\mu:S\to\mathfrak{g}^\ast$ be a moment map. Let $f$ be an equivariant Hermitian isometry of $S$. Then we show that the action on $S$ is naturally lifted to a Hamiltonian action on $S_f$ if and only if $\mu(f(p))=\mu(p)$ holds for some $p\in S$. In this situation, we prove that the K\"{a}hler/coK\"{a}hler reduction procedures are compatible with the mapping torus procedure.

In \autoref{dynamical}, we interpret our coK\"{a}hler reduction theorem from the physical viewpoint. In short, our result suggests that we can reduce time-dependent dynamical systems preserving the property that the flows of the system are geodesics.

\vspace{0.2in}
\noindent{\bf Acknowledgments.} The author would like to express his gratitude to his supervisor Professor Ryushi Goto for his encouragement and many fruitful discussions. This work was supported by JSPS KAKENHI Grant Number JP23KJ1487.

\section{Cosymplectic reduction theorem}\label{cosymplectic}

An \textit{almost cosymplectic structure} on $(2n+1)$-dimensional manifold $M$ is a pair of $\eta\in\Omega^1(M)$ and $\omega\in\Omega^2(M)$ such that $\eta\wedge\omega^n\neq0$. On an almost cosymplectic manifold $(M,\eta,\omega)$ there is a unique vector field $\xi$ which satisfies 
        \[\omega(\xi,-)=0,\qquad\eta(\xi)=1.\]
$\xi$ is called the Reeb vector field of $(M,\eta,\omega)$. Moreover, we have an isomorphism of $C^\infty(M)$-modules $\flat:\mathfrak{X}(M)\to \Omega^1(M)$ defined by $\flat(X)=\omega(X,-)+\eta(X)\eta$. Conversely, a pair $(\eta,\omega)$ is an almost cosymplectic structure if and only if the map $\flat:\mathfrak{X}(M)\to \Omega^1(M)$ defined as above is an isomorphism and there is a vector field $\xi$ which satisfies the above conditions.
    
An almost cosymplectic structure $(\eta,\omega)$ is called a contact structure when $\omega=d\eta$.
On the other hand, an almost cosymplectic structure $(\eta,\omega)$ is called a \textit{cosymplectic structure} when $d\eta=0,\ d\omega=0$.

For a contact structure $\eta\in\Omega^1(M)$, the distribution $\textrm{Ker}\eta$ is completely non-integrable. On the other hand, for a cosymplectic structure $(\eta,\omega)$, the distribution $\textrm{Ker}\eta$ is integrable since $\eta$ is closed. Therefore, contact structures and cosymplectic structures are two classes of almost cosymplectic structures which are polar opposites of each other.

For every function $f\in C^\infty(M)$ on a cosymplectic manifold $M$, we associate a vector field $X_f$ by
\[X_f=\flat^{-1}(df-\xi(f)\eta).\]
$X_f$ is called the Hamiltonian vector field of $f$. This condition is equivalent to 
\[\omega(X_f,-)=df-\xi(f)\eta,\qquad\eta(X_f)=0.\]

Let $(M,\eta,\omega)$ be a cosymplectic manifold and $G$ a Lie group acts on $M$ from left. We suppose that the action preserves $\eta,\omega$, i.e., $L_g^\ast \eta=\eta,\ L_g^\ast \omega=\omega$. Denote the Lie algebra of $G$ as $\mathfrak{g}$. Albert \cite{albert1989theoreme} defined the notion of moment maps on cosymplectic manifolds:
\begin{definition}
    A smooth map $\mu:M\to\mathfrak{g}^\ast$ is called a moment map when the following conditions are satisfied:
    \begin{itemize}
        \item $\mu$ is equivariant, i.e., $\mu(gp)=\textrm{Ad}^\ast_g\mu(p)$ holds for any $p\in M$ and $g\in G$.
        \item For any $A\in\mathfrak{g}$, the induced vector field $A^\ast\in\mathfrak{X}(M)$ is the Hamiltonian vector field of a function $\mu^A:M\to\mathbb{R}$ defined by $\mu^A(p)=(\mu(p))(A)$,
        \item For the Reeb vector field $\xi$ and any $A\in\mathfrak{g}$, $d\mu^A(\xi)=0$ holds.$\End$
    \end{itemize}
\end{definition}

The action of $G$ is said to be Hamiltonian if there is a moment map. Now we assume that there is a Hamiltonian action of $G$ on $(M,\eta,\omega)$ which is free and proper. Let $\zeta\in\mathfrak{g}^\ast$ be a regular value of a moment map $\mu:M\to\mathfrak{g}^\ast$. Since $\mu$ is equivariant, the isotropy group $G_\zeta$ acts on $\mu^{-1}(\zeta)$. Let $M^\zeta:=\mu^{-1}(\zeta)/G_\zeta$ and $\pi:\mu^{-1}(\zeta)\to M^\zeta$ be the natural projection.

\begin{theorem}[Albert \cite{albert1989theoreme}]\label{-3}
    There is a unique cosymplectic structure $(\eta^\zeta,\omega^\zeta)$ on $M^\zeta$ which satisfies $\pi^\ast\eta^\zeta=\eta|_{\mu^{-1}(\zeta)},\ \pi^\ast\omega^\zeta=\omega|_{\mu^{-1}(\zeta)}$.
\end{theorem}
\begin{proof}
    Since $d\mu^A(X)=\omega(A^\ast,X)$, for any $p\in\mu^{-1}(\zeta)$ we have 
    \begin{equation}\label{-2.1}
        T_p\mu^{-1}(\zeta)=\{X_p\in T_p M\mid\omega(A_p^\ast,X_p)=0,\ A\in\mathfrak{g}\}.
    \end{equation}    
    Let $\mathfrak{g}_\zeta$ be the Lie algebra of $G_\zeta$. For any $p\in\mu^{-1}(\zeta)$ and $g\in G_\zeta$ we have $\mu(L_g(p))=\zeta$, and by differentiating this we obtain
    \[\mathfrak{g}_\zeta=\{B\in\mathfrak{g}\mid d\mu(B^\ast_p)=0\}\]
    for any $p\in\mu^{-1}(\zeta)$. Hence, if we define $\mathfrak{g}_p:=\{A_p^\ast\mid A\in\mathfrak{g}\}$ and $\mathfrak{g}_{\zeta,p}:=\{B_p^\ast\mid B\in\mathfrak{g}_\zeta\}$, we obtain
    \begin{equation}\label{-2.01}
        \mathfrak{g}_{\zeta,p}=\mathfrak{g}_p\cap T_p\mu^{-1}(\zeta).
    \end{equation}

    We see that $\eta|_{\mu^{-1}(\zeta)}$ and $\omega|_{\mu^{-1}(\zeta)}$ are basic with respect to the fibration $\pi:\mu^{-1}(\zeta)\to M^\zeta$. For any $B\in\mathfrak{g}_\zeta$, we have $\eta|_{\mu^{-1}(\zeta)}(B^\ast)=0$. In addition, since $d\eta|_{\mu^{-1}(\zeta)}=0$,
    \[\mathcal{L}_{B^\ast}\eta|_{\mu^{-1}(\zeta)}=d\iota_{B^\ast}\eta|_{\mu^{-1}(\zeta)}+\iota_{B^\ast}d\eta|_{\mu^{-1}(\zeta)}=0\]
    holds, and thus $\eta|_{\mu^{-1}(\zeta)}$ is basic. On the other hand, (\ref{-2.1}) implies that $\omega|_{\mu^{-1}(\zeta)}(B^\ast_p,X_p)=0$ for any $B\in\mathfrak{g}_\zeta$ and $X_p\in T_p\mu^{-1}(\zeta)$, so similarly $\omega|_{\mu^{-1}(\zeta)}$ is basic. Then we obtain $\eta^\zeta$ and $\omega^\zeta$. Moreover, they are closed since $\eta,\omega$ are closed and $\pi^\ast$ is injective.

    All that remains is to prove that $(\eta^\zeta,\omega^\zeta)$ is an almost cosymplectic structure. We have $\xi_p\in T_p\mu^{-1}(\zeta)$ for any $p\in\mu^{-1}(\zeta)$, and $L_g^\ast \eta=\eta,\ L_g^\ast \omega=\omega$ implies $(L_g)_\ast\xi_p=\xi_{L_g(p)}$. So we can define a vector field $\xi^\zeta=d\pi(\xi|_{\mu^{-1}(\zeta)})$ on $M^\zeta$, and we have $\eta^\zeta(\xi^\zeta)=1,\ \omega^\zeta(\xi^\zeta,-)=0$.

    Lastly, we prove that the map $\flat^\zeta:TM^\zeta\to T^\ast M^\zeta;\ X^\zeta\mapsto\iota_{X^\zeta}\omega^\zeta+\iota_{X^\zeta}\eta^\zeta\eta^\zeta$ is an isomorphism. Suppose that $X^\zeta\in TM^\zeta$ satisfies $\flat^\zeta(X^\zeta)=0$. Take $p\in\mu^{-1}(\zeta)$ and let $x=\pi(p)$. We can take $X_p\in T_p\mu^{-1}(\zeta)$ such that $d\pi(X_p)=X_x^\zeta$. Then 
    \[\omega(X_p,Y_p)+\eta(X_p)\eta(Y_p)=0\]
    holds for any $Y_p\in T_p\mu^{-1}(\zeta)$, and this implies
    \begin{equation}\label{-2.001}
        \eta(X_p)=0,\qquad\omega(X_p,Y_p)=0.
    \end{equation}
    The almost cosymplectic structure $(\eta,\omega)$ on $M$ gives a decomposition
    \[T_pM=\mathbb{R}\langle\xi_p\rangle\oplus\textrm{Ker}\eta_p\]
    and $\omega_p|_{\textrm{Ker}\eta_p}$ is non-degenerate. We define $E_p=\textrm{Ker}\eta_p\cap T_p\mu^{-1}(\zeta)$. Then by (\ref{-2.1}) and (\ref{-2.001}) we obtain
    \[X_p\in(E_p)^{\omega_p|_{\textrm{Ker}\eta_p}}=((\mathfrak{g}_p)^{\omega_p|_{\textrm{Ker}\eta_p}})^{\omega_p|_{\textrm{Ker}\eta_p}}=\mathfrak{g}_p,\]
    where $V^{\omega_p|_{\textrm{Ker}\eta_p}}$ denotes the orthogonal complement of $V\subset\textrm{Ker}\eta_p$ with respect to $\omega_p|_{\textrm{Ker}\eta_p}$. Now we can conclude that $X_p\in\mathfrak{g}_{\zeta,p}$ by (\ref{-2.01}) and thus $X_x^\zeta=d\pi(X_p)=0.$
\end{proof}

\section{CoK\"{a}hler case}\label{cokahler}

Let $g$ be a Riemannian metric and $(\varphi,\xi,\eta)$ an almost contact structure on $M$, i.e., a triplet of $\varphi\in\textrm{End}(TM)$, a vector field $\xi$ and a 1-form $\eta$, which satisfies $\varphi^2=-\rm{id}+\eta\otimes\xi,\ \eta(\xi)=1$. Then a quartets $(g,\varphi,\xi,\eta)$ is called an almost contact metric structure if 
\[g(\varphi X,\varphi Y)=g(X,Y)-\eta(X)\eta(Y)\]
holds.

For a almost contact metric structure $(g,\varphi,\xi,\eta)$, we obtain an almost cosymplectic structure $(\eta,\omega)$, where $\omega(X,Y)=g(X,\varphi Y)$.

\begin{definition}
    An \textit{almost coK\"{a}hler structure} is an almost contact metric structure $(g,\varphi,\xi,\eta)$ such that the induced almost cosymplectic structure $(\eta,\omega)$ is a cosymplectic structure.
   
    An almost coK\"{a}hler structure $(g,\varphi,\xi,\eta)$ is said to be a \textit{coK\"{a}hler structure} if the almost contact structure $(\varphi,\xi,\eta)$ is normal, i.e., the Nijenhuis tensor $N_\varphi$ of $\varphi$ satisfies $N_\varphi=-2d\eta\otimes\xi$.$\End$
\end{definition}

If we replace the cosymplectic condition $d\eta=d\omega=0$ with $\omega=d\eta$ in the definition above, then $(g,\varphi,\xi,\eta)$ is called a Sasakian structure.

Let $(M,g,\varphi,\xi,\eta)$ be a coK\"{a}hler manifold and $G$ a Lie group. Suppose that there is a Hamiltonian action of $G$ on a cosymplectic manifold $(M,\eta,\omega)$, and let $\zeta\in\mathfrak{g}^\ast$ be a regular value of a moment map $\mu:M\to\mathfrak{g}^\ast$. Moreover, we assume that the action preserves $\varphi$ (and hence preserves the metric $g$) and $\zeta$ is central, i.e., $G_\zeta=G$ holds.

\begin{lemma}\label{-2.0003}
    Let $p\in\mu^{-1}(\zeta)$. Then there is a subspace $H_p\subset T_pM$ and a decomposition
    \begin{equation}\label{-2.0002}
        T_p M=H_p\oplus\mathfrak{g}_{p}\oplus\varphi_p(\mathfrak{g}_p),
    \end{equation}
    which is orthogonal with respect to the metric $g$. Moreover, $H_p$ is invariant under $\varphi_p$.
\end{lemma}
\begin{proof}
    For any $v\in T_p\mu^{-1}(\zeta)$ and $A^\ast_p\in\mathfrak{g}_p$, we have
    \[g(v,\varphi_pA^\ast_p)=\omega(v,A^\ast_p)=-(d\mu^A)(v)=0.\]
    In addition, since $\varphi_p$ is an isomorphism on $\textrm{Ker}\eta_p$, we have $\dim\varphi_p(\mathfrak{g}_p)=\dim G$ and thus we obtain a decomposition
    \[T_p M=T_p\mu^{-1}(\zeta)\oplus\varphi_p(\mathfrak{g}_p).\]
    We define $H_p$ as the $g$-orthogonal complement of $\mathfrak{g}_{p}$ in $T_p\mu^{-1}(\zeta)$. Then we obtain the decomposition (\ref{-2.0002}). Note that $\xi_p\in H_p$ holds since $g(\xi_p,A^\ast_p)=\eta(A^\ast_p)=0$. Therefore the decomposition (\ref{-2.0002}) implies $\varphi_p(H_p)=H_p$ since $\varphi_p^2(v)=-v$ holds for $v\in\textrm{Ker}\eta_p$.
\end{proof}

\begin{lemma}\label{-2.00011}
    $H_p$ is invariant under the action of $G$, namely, $(L_h)_\ast(H_p)=H_{hp}$ holds for any $h\in G$ and $p\in\mu^{-1}(\zeta)$.
\end{lemma}
\begin{proof}
    For any $A^\ast_p\in\mathfrak{g}_p$, we have 
    \[(L_h)_\ast(\varphi_pA^\ast_p)=\varphi_{hp}(L_h)_\ast A^\ast_p=\varphi_{hp}(\textrm{Ad}_hA)^\ast_{hp}\]
    and thus $(L_h)_\ast(\varphi_h(\mathfrak{g}_p))=\varphi_{hp}(\mathfrak{g}_{hp})$. Hence the decomposition (\ref{-2.0002}) implies $(L_h)_\ast v\in T_{hp}\mu^{-1}(\zeta)$ for $v\in H_p$. Moreover, $(L_h)_\ast v$ is orthogonal to $\mathfrak{g}_{hp}$ since the action of $G$ preserves the metric $g$, therefore we obtain $(L_h)_\ast(H_p)=H_{hp}$.
\end{proof}

Now we obtain the following reduction theorem.
\begin{theorem}\label{100}
    Let $(M,g,\varphi,\xi,\eta)$ be a coK\"{a}hler manifold with the underlying cosymplectic structure $(\eta,\omega)$. Suppose that there is a free and proper Hamiltonian action of a Lie group $G$ on $(M,\eta,\omega)$ which preserves $\varphi$. Let $\mu:M\to\mathfrak{g}^\ast$ be a moment map and $\zeta\in\mathfrak{g}^\ast$ a central and regular value of $\mu$.
    Then $M^\zeta:=\mu^{-1}(\zeta)/G$ admits a coK\"{a}hler structure $(g^\zeta,\varphi^\zeta,\xi^\zeta,\eta^\zeta)$. Moreover, the underlying cosymplectic manifold of $(M^\zeta,g^\zeta,\varphi^\zeta,\xi^\zeta,\eta^\zeta)$ is the cosymplectic quotient $(M^\zeta,\eta^\zeta,\omega^\zeta)$.
\end{theorem}
\begin{proof}
    For any $p\in\mu^{-1}(\zeta)$, the map $(d\pi)_p|_{H_p}:H_p\to T_{\pi(p)}M^\zeta$ is an isomorphism.
   We define a Riemannian metric $g^\zeta$ on $M^\zeta$ and $\varphi^\zeta\in{\rm End}(TM^\zeta)$ as pushforwards by $(d\pi)_p|_{H_p}$, i.e.,
   \[g^\zeta_{\pi(p)}(X_{\pi(p)},Y_{\pi(p)})=g_p(\widetilde{X}_p,\widetilde{Y}_p),\]
   \[\varphi^\zeta_{\pi(p)}(X_{\pi(p)})=(d\pi)_p(\varphi_p(\widetilde{X}_p)),\]
   where $\widetilde{X}_p,\widetilde{Y}_p\in H_p$ are vectors which satisfies
   \[(d\pi)_p(\widetilde{X}_p)=X_{\pi(p)},\qquad(d\pi)_p(\widetilde{Y}_p)=Y_{\pi(p)}.\]
   
   We check that $g^\zeta_{\pi(p)},\varphi^\zeta_{\pi(p)}$ are independent of the choice of $p$. For $h\in G$, we have $(d\pi)_{hp}(L_h)_\ast\widetilde{X}_p=(d\pi)_p\widetilde{X}_p$, and thus \autoref{-2.00011} implies $\widetilde{X}_{hp}=(L_h)_\ast\widetilde{X}_p$. Hence we obtain
   \[g_p(\widetilde{X}_p,\widetilde{Y}_p)=g_{hp}((L_h)_\ast\widetilde{X}_p,(L_h)_\ast\widetilde{Y}_p)=g_{hp}(\widetilde{X}_{hp},\widetilde{Y}_{hp}),\]
   \[\begin{split}
    (d\pi)_p(\varphi_p(\widetilde{X}_p))&=(d\pi)_{hp}(L_h)_\ast(\varphi_p(\widetilde{X}_p))\\
   &=(d\pi)_{hp}(\varphi_{hp}((L_h)_\ast\widetilde{X}_p))\\
   &=(d\pi)_{hp}(\varphi_{hp}(\widetilde{X}_{hp})),
   \end{split}\]
   therefore $g^\zeta_{\pi(p)},\varphi^\zeta_{\pi(p)}$ are well-defined (the metric $g^\zeta$ is called the quotient metric on $M^\zeta$).

   Let $(M^\zeta,\eta^\zeta,\omega^\zeta)$ be the cosymplectic quotient and $\xi^\zeta$ the Reeb vector field. Then $\eta^\zeta(\xi^\zeta)=1$ holds. Moreover, since $\varphi_p$ preserves $H_p$, we have $\widetilde{\varphi^\zeta(X)}=\varphi(\widetilde{X})$, and thus we obtain
   \[\begin{split}
    \varphi^\zeta(\varphi^\zeta(X))&=d\pi(\varphi\widetilde{\varphi^\zeta(X)})\\
   &=d\pi(\varphi\varphi(\widetilde{X}))\\
   &=d\pi(-\widetilde{X}+\eta(\widetilde{X})\xi)\\
   &=-X+\eta^\zeta(X)\xi^\zeta,
   \end{split}\]
   so $(\varphi^\zeta,\xi^\zeta,\eta^\zeta)$ is an almost contact structure on $M^\zeta$.
   
   We can easily check the compatibility of $g^\zeta,\varphi^\zeta$ with $\eta^\zeta,\omega^\zeta$, i.e.,
   \[g^\zeta(\varphi^\zeta X,\varphi^\zeta Y)=g^\zeta(X,Y)-\eta^\zeta(X)\eta^\zeta(Y),\]
   \[\omega^\zeta(X,Y)=g^\zeta(X,\varphi^\zeta Y).\]

   Lastly, we prove that $(g^\zeta,\varphi^\zeta,\xi^\zeta,\eta^\zeta)$ is a coK\"{a}hler structure. It was proved in \cite{blair1966theory} that an almost contact metric structure $(g,\varphi,\xi,\eta)$ is coK\"{a}hler if and only if $\nabla\varphi=0$, where $\nabla$ is the Levi-Civita connection of the metric $g$.
   Let $\nabla$ be the Levi-Civita connection of the metric on $M$. Using a general property of quotient metrics, we can compute the Levi-Civita connection $\nabla^\zeta$ of $g^\zeta$ by
   \[\nabla^\zeta_XY=d\pi(\textrm{pr}_H(\nabla_{\widetilde{X}}\widetilde{Y})),\]
   where $\widetilde{X},\widetilde{Y}$ are extended to a neighborhood of $\mu^{-1}(\zeta)$ and $\textrm{pr}_H:TM\to H$ denotes the orthogonal projection. Then since $\nabla\varphi=0$ we have
   \[\begin{split}
    \widetilde{\nabla^\zeta_X\varphi^\zeta Y}&=\textrm{pr}_H(\nabla_{\widetilde{X}}\varphi\widetilde{Y})=\textrm{pr}_H(\varphi\nabla_{\widetilde{X}}\widetilde{Y})\\
   &=\varphi\textrm{pr}_H(\nabla_{\widetilde{X}}\widetilde{Y})=\varphi(\widetilde{\nabla^\zeta_XY})\\
   &=\widetilde{\varphi^\zeta(\nabla^\zeta_XY)}.
   \end{split}\]
   Hence $\nabla^\zeta_X\varphi^\zeta(Y)=\varphi^\zeta(\nabla^\zeta_XY)$ and thus $\nabla^\zeta\varphi^\zeta=0$.
\end{proof}

\section{3-cosymplectic case}\label{3-cosymplectic}
In this section, we prove a reduction theorem for 3-cosymplectic manifolds. First we recall the definition of 3-cosymplectic structures.
\begin{definition}\label{-2.0001}
    A \textit{3-cosymplectic structure} on $M$ is a quartet $(g,(\varphi_i,\xi_i,\eta_i)_{i=1,2,3})$ of a Riemannian metric $g$ and three almost contact structures $(\varphi_i,\xi_i,\eta_i)_{i=1,2,3}$ on $M$ such that each $(g,\varphi_i,\xi_i,\eta_i)$ is coK\"{a}hler and
    \[\varphi_\gamma=\varphi_\alpha\varphi_\beta-\eta_\beta\otimes\xi_\alpha=-\varphi_\beta\varphi_\alpha+\eta_\alpha\otimes\xi_\beta,\]
    \[\xi_\gamma=\varphi_\alpha\xi_\beta=-\varphi_\beta\xi_\alpha,\]
    \[\eta_\gamma=\varphi_\beta^\ast\eta_\alpha=-\varphi_\alpha^\ast\eta_\beta\]
    holds for any even permutation $(\alpha,\beta,\gamma)$ of $\{1,2,3\}$.
    $\End$
\end{definition}

\begin{remark}
    This notion should be called ``3-coK\"{a}hler structures'', but since the name ``3-cosymplectic'' has become established, we will follow it here as well. If a quartet $(g,(\varphi_i,\xi_i,\eta_i)_{i=1,2,3})$ satisfies all conditions of \autoref{-2.0001} except the normality of each $(\varphi_i,\xi_i,\eta_i)$, the normality of them automatically follows (see \cite{pastore2004riemannian}).
    $\End$
\end{remark}

\begin{lemma}\label{0}
    Let $(M,g,(\varphi_i,\xi_i,\eta_i)_{i=1,2,3})$ be a 3-cosymplectic manifold and $\omega_2$ a 2-form defined by $\omega_2(X,Y)=g(X,\varphi_2Y)$. Then $\omega_2(X,\xi_1)=-\eta_3(X)$ holds.
\end{lemma}
\begin{proof}
    Since $\eta_\alpha(\xi_\gamma)=\eta_\alpha(\varphi_\alpha\xi_\beta)=0$, we obtain
    \[\begin{split}
        \omega_2(\varphi_3X,Y)&=g(\varphi_3X,\varphi_2Y) \\
       &=g(\varphi_3\varphi_3X,\varphi_3\varphi_2Y)+\eta_3(\varphi_3X)\eta_3(\varphi_2Y) \\
       &=g(-X+\eta_3(X)\xi_3,-\varphi_2\varphi_3Y+\eta_3(Y)\xi_2+\eta_2(Y)\xi_3) \\
       &=g(X,\varphi_2\varphi_3Y)-\eta_3(Y)\eta_2(X)-\eta_2(Y)\eta_3(X) \\
       &\ \ \ \ -\eta_3(X)\eta_3(\varphi_2\varphi_3Y)+\eta_3(X)\eta_3(Y)\eta_3(\xi_2)+\eta_3(X)\eta_2(Y)\eta_3(\xi_3) \\
       &=g(X,\varphi_2\varphi_3Y)-\eta_2(X)\eta_3(Y)-\eta_3(X)\eta_3(-\varphi_3\varphi_2Y+\eta_3(Y)\xi_2+\eta_2(Y)\xi_3) \\
       &=g(X,\varphi_2\varphi_3Y)-\eta_2(X)\eta_3(Y)-\eta_3(X)\eta_2(Y) \\
       &=\omega_2(X,\varphi_3Y)-\eta_2(X)\eta_3(Y)-\eta_3(X)\eta_2(Y).
   \end{split}\]
   Therefore we have
    \[\begin{split}
        \omega_2(X,\xi_1)&=-\omega_2(X,\varphi_3\xi_2) \\
       &=-\omega_2(\varphi_3X,\xi_2)-\eta_2(X)\eta_3(\xi_2)-\eta_3(X)\eta_2(\xi_2) \\
       &=-\eta_3(X).
   \end{split}\]
\end{proof}

Let $G$ be a Lie group acting on a 3-cosymplectic manifold $(M,g,(\varphi_i,\xi_i,\eta_i)_{i=1,2,3})$. Suppose that the action is Hamiltonian with respect to all three cosymplectic structures $(\eta_i,\omega_i)_{i=1,2,3}$, and let $\mu_i:M\to\mathfrak{g}^\ast\ (i=1,2,3)$ be moment maps.

\begin{lemma}\label{1}
    Let $\zeta_2,\zeta_3\in\mathfrak{g}^\ast$ be regular values of $\mu_2,\mu_3$, respectively. Suppose that submanifolds $\mu_2^{-1}(\zeta_2)$ and $\mu_3^{-1}(\zeta_3)$ intersect transversally. Then $N:=\mu^{-1}_2(\zeta_2)\cap\mu^{-1}_3(\zeta_3)$ is an almost contact submanifold of $(M,\varphi_1,\xi_1,\eta_1)$.
\end{lemma}
\begin{proof}
    Using \autoref{0}, we have
    \[d\mu_2^A(\xi_1)=\omega_2(A^\ast,\xi_1)=-\eta_3(A^\ast)=0,\]
    and similarly $d\mu_3^A(\xi_1)=0$ holds, thus $\xi_1\in TN$.
    Next, we check that $\varphi_1(TN)\subset TN$. Suppose that $d\mu_3(X)=0.$ Then we obtain
    \[\begin{split}
        0&=d\mu_3^A(X)=\omega_3(A^\ast,X)=g(A^\ast,\varphi_3X) \\
       &=g(A^\ast,-\varphi_2\varphi_1X+\eta_1(X)\xi_2) \\
       &=-g(A^\ast,\varphi_2\varphi_1X)+\eta_1(X)\eta_2(A^\ast) \\
       &=-\omega_2(A^\ast,\varphi_1X)=-d\mu_2^A(\varphi_1X).
   \end{split}\]
    Similarly, we can see that $d\mu_2(X)=0$ implies $d\mu^A_3(\varphi_1X)=0$, hence $\varphi_1X\in TN$ holds for any $X\in TN$.
\end{proof}

We define a 3-cosymplectic moment map $\mu:M\to\mathfrak{g}^\ast\otimes{\rm Im}\mathbb{H}$ by
\[\mu=\mu_1i+\mu_2j+\mu_3k,\]
where $i,j,k$ are generators of ${\rm Im}\mathbb{H}$. Let $G$ act on $\mathfrak{g}^\ast\otimes{\rm Im}\mathbb{H}$ by the tensor representation of the coadjoint action and the trivial action on ${\rm Im}\mathbb{H}$. Then for any $\zeta=\zeta_1i+\zeta_2j+\zeta_3k\in\mathfrak{g}^\ast\otimes{\rm Im}\mathbb{H}$, the $G_\zeta$-action preserves $\mu^{-1}(\zeta)=\mu^{-1}_1(\zeta_1)\cap\mu^{-1}_2(\zeta_2)\cap\mu^{-1}_3(\zeta_3)$.
\begin{theorem}\label{200}
    Let $(M,g,(\varphi_i,\xi_i,\eta_i)_{i=1,2,3})$ be a 3-cosymplectic manifold with underlying cosymplectic structures $(\eta_i,\omega_i)_{i=1,2,3}$. Suppose that there is a free and proper action of a Lie group $G$ on $M$ which is Hamiltonian with respect to all three cosymplectic structures $(\eta_i,\omega_i)_{i=1,2,3}$ and preserves $(\varphi_i)_{i=1,2,3}$. Let $\mu:M\to\mathfrak{g}^\ast\otimes{\rm Im}\mathbb{H}$ be a 3-cosymplectic moment map and $\zeta\in\mathfrak{g}^\ast\otimes{\rm Im}\mathbb{H}$ a central and regular value of $\mu$. Then $M^\zeta:=\mu^{-1}(\zeta)/G$ inherits the 3-cosymplectic structure of $M$. 
\end{theorem}
\begin{proof}
    \autoref{1} implies that there is an almost contact structure $(\varphi_1|_N,\xi_1|_N,\eta_1|_N)$ on $N$. Therefore from a result in \cite{ludden1970submanifolds}, we can see that $(N,g|_N,\varphi_1|_N,\xi_1|_N,\eta_1|_N)$ is a coK\"{a}hler manifold. 
    
    The action of $G$ on $N$ preserves the coK\"{a}hler structure, and $\mu_1|_N$ is a moment map for this action. So \autoref{100} implies that there is a coK\"{a}hler structure $(g_1^\zeta,\varphi_1^\zeta,\xi_1^\zeta,\eta_1^\zeta)$ on $M^\zeta=(N\cap\mu^{-1}_1(\zeta_1))/G$. Similarly, we obtain two more coK\"{a}hler structures $(g_2^\zeta,\varphi_2^\zeta,\xi_2^\zeta,\eta_2^\zeta)$ and $(g_3^\zeta,\varphi_3^\zeta,\xi_3^\zeta,\eta_3^\zeta)$ on $M^\zeta$.
    
    Each Riemannian metric $g_i^\zeta\ (i=1,2,3)$ coincides with the quotient metric of the principal bundle $\mu^{-1}(\zeta)\to M^\zeta$, so three Riemannian metrics $g_1^\zeta,g_2^\zeta,g_3^\zeta$ are the same, and thus the three coK\"{a}hler structures constitutes a 3-cosymplectic structure on $M^\zeta$.
\end{proof}

\section{CoK\"{a}hler/3-cosymplectic reductions and cone constructions}\label{cone}

When $M$ is endowed with a geometric structure we are studying, it induces a geometric structure on the cone $C(M):=M\times\mathbb{R}$ as shown in the table below (in the case that $M$ is hyperK\"{a}hler, we use $C^3(M):=M\times\mathbb{R}^3$ instead of the cone). In this section, we prove that reduction procedures are compatible with these constructions.

\vspace{0.1in}
\begin{center}
    \begin{threeparttable}[htbp]
        \begin{tabular}{|c|c|}
            \hline
            Structure on the base&Induced structure on the cone\\ 
          \hline
          \hline
          K\"{a}hler&CoK\"{a}hler\\
          \hline
          CoK\"{a}hler&K\"{a}hler\\ 
          \hline
          HyperK\"{a}hler&3-cosymplectic\\ 
          \hline
          3-cosymplectic&HyperK\"{a}hler\\ 
          \hline
        \end{tabular}
    \end{threeparttable}
\end{center}

\vspace{0.1in}
Let $(M,h,J)$ be a K\"{a}hler manifold ($h$ denotes the Riemannian metric and $J$ denotes the complex structure). Then $C(M):=M\times\mathbb{R}$ admits a natural coK\"{a}hler structure $(g,\varphi,\xi,\eta)$ defined by 
\[g=h+dt^2,\qquad\varphi\Bigl(X,f\displaystyle\frac{\partial}{\partial t}\Bigr)=(JX,0),\]

\[\xi=\displaystyle\frac{\partial}{\partial t},\qquad\eta=dt,\]
where $t$ is the coordinate of $\mathbb{R}$ and $f\in C^\infty(C(M))$.

\begin{example}[\cite{fujimoto1974cosymplectic}]\label{1.1}
    Let $f$ be a Hermitian isometry on $(M,h,J)$. Define an action of $\mathbb{Z}$ on $C(M)$ by $k\cdot(p,t):=(f^k(p),t+k)$. This action is free and properly discontinuous, so $C(M)/\mathbb{Z}$ is a smooth manifold. Then $C(M)/\mathbb{Z}$ inherits a coK\"{a}hler structure. This is a coK\"{a}hler quotient for a trivial moment map on $C(M)$. $C(M)/\mathbb{Z}$ is diffeomorphic to the mapping torus of $M$ with respect to $f$, and we will discuss the coK\"{a}hler structure on it in detail later.$\End$
\end{example}

Assume that there is a Hamiltonian action of $G$ on $M$ preserving $J$, and let $\mu:M\to\mathfrak{g}^\ast$ be a moment map. We define an action of $G$ on $C(M)$ by $g\cdot(p,t):=(gp,t)$. Then this action preserves the coK\"{a}hler structure on $C(M)$, and is a Hamiltonian action whose moment map is $\widetilde{\mu}=\mu\circ\textrm{pr}_M.$
\begin{proposition}\label{2}
     When a K\"{a}hler quotient $\mu^{-1}(\zeta)/G$ is defined for $\zeta\in\mathfrak{g}^\ast$, a coK\"{a}hler quotient $\widetilde{\mu}^{-1}(\zeta)/G$ is also defined. Moreover, $C(\mu^{-1}(\zeta)/G)$ and $\widetilde{\mu}^{-1}(\zeta)/G$ are equivalent as coK\"{a}hler manifolds.
\end{proposition}
\begin{proof}
Clearly $C(\mu^{-1}(\zeta)/G)$ and $\widetilde{\mu}^{-1}(\zeta)/G$ are diffeomorphic, and the diagram
\[\xymatrix{\widetilde{\mu}^{-1}(\zeta)\ar[d]_-{\textrm{pr}}\ar[r]^-{\widetilde{\pi}}&C(\mu^{-1}(\zeta)/G)\ar[d]^-{\textrm{pr}}\\
\mu^{-1}(\zeta)\ar[r]_-{\pi}&{\mu}^{-1}(\zeta)/G
    }\]
commutes, where $\textrm{pr}$ denotes natural projections and $\widetilde{\pi}$ is the projection of coK\"{a}hler reduction.
We orthogonally decompose $T_p\mu^{-1}(\zeta)$ and $T_{(p,t)}\widetilde{\mu}^{-1}(\zeta)$ as 
\[T_p\mu^{-1}(\zeta)=H_p\oplus\mathfrak{g}_p,\]
\[T_{(p,t)}\widetilde{\mu}^{-1}(\zeta)=H'_{(p,t)}\oplus\mathfrak{g}_{(p,t)},\]
respectively. Then we can easily check that $(\textrm{pr})_\ast(\mathfrak{g}_{(p,t)})=\mathfrak{g}_p$ and thus obtain the following commutative diagram.
\[\xymatrix{H'_{(p,t)}\ar[d]_-{(\textrm{pr})_\ast}\ar[r]^-{d\widetilde{\pi}}_-{\simeq}&T_{\widetilde{\pi}(p,t)}C(M)^\zeta\ar[d]^-{(\textrm{pr})_\ast}\\
H_p\ar[r]_-{d\pi}^-{\simeq}&T_{\pi(p)}M^\zeta
}\]

We denote the lift of $X\in T_{\pi(p)}M^\zeta$ and $(X,f\frac{\partial}{\partial r})\in T_{\widetilde{\pi}(p,t)}C(M)^\zeta$ as $\widetilde{X},\widetilde{(X,f\frac{\partial}{\partial r})}$, respectively. Let $r$ be the coordinate of $\mathbb{R}$ in $C({\mu}^{-1}(\zeta)/G)$. Since $d\pi(\textrm{pr})_\ast\widetilde{\frac{\partial}{\partial r}}=(\textrm{pr})_\ast{\frac{\partial}{\partial r}}=0$, we get $(\textrm{pr})_\ast\widetilde{\frac{\partial}{\partial r}}=0$. Hence $\widetilde{\frac{\partial}{\partial r}}=a\xi$ holds for some $a\in\mathbb{R}^\times$. We normalize the coordinate $r$ to satisfy $\widetilde{\frac{\partial}{\partial r}}=\xi$. Then we obtain $\xi^\zeta=\frac{\partial}{\partial r}$ and $\eta^\zeta=dr$.

Since $\widetilde{\frac{\partial}{\partial r}}=\xi$, 
\[\varphi\widetilde{\Bigl(X,f\displaystyle\frac{\partial}{\partial r}\Bigr)}=\varphi\Bigl(\widetilde{X},\widetilde{f\displaystyle\frac{\partial}{\partial r}}\Bigr)=(J\widetilde{X},0)\]
holds. Hence from the diagram above, we obtain
\[\varphi^\zeta\Bigl(X,f\displaystyle\frac{\partial}{\partial r}\Bigr)=d\widetilde{\pi}\varphi\widetilde{\Bigl(X,f\displaystyle\frac{\partial}{\partial r}\Bigr)}=(d\pi(J\widetilde{X}),0)=(J^\zeta X,0).\]

We can also see that $g^\zeta=h^\zeta+dr^2$, therefore the coK\"{a}hler structure $(g^\zeta,\varphi^\zeta,\xi^\zeta,\eta^\zeta)$ coincides with one obtained by the cone construction 
\[C({\mu}^{-1}(\zeta)/G)=({\mu}^{-1}(\zeta)/G)\times\mathbb{R}.\]
\end{proof}

\vspace{0.1in}
Conversely, for a given coK\"{a}hler manifold $(M,g,\varphi,\xi,\eta)$, we can define a K\"{a}hler structure $(h,J)$ on the cone $C(M)$ by
\[h=g+dt^2,\qquad J\Bigl(X,f\displaystyle\frac{\partial}{\partial t}\Bigr)=\Bigl(\varphi X-f\xi,\eta(X)\displaystyle\frac{\partial}{\partial t}\Bigr).\]
Assume that there is a Hamiltonian action of $G$ on $M$ preserving the coK\"{a}hler structure, and let $\mu:M\to\mathfrak{g}^\ast$ be a moment map. We define an action of $G$ on $C(M)$ in the same way as in \autoref{2}. Then the action preserves the K\"{a}hler structure on $C(M)$ and $\widetilde{\mu}=\mu\circ\textrm{pr}_M$ is a moment map. We obtain the following.
\begin{proposition}\label{2.01}
    When a coK\"{a}hler quotient $\mu^{-1}(\zeta)/G$ is defined for $\zeta\in\mathfrak{g}^\ast$, a K\"{a}hler quotient $\widetilde{\mu}^{-1}(\zeta)/G$ is also defined. Moreover, $C(\mu^{-1}(\zeta)/G)$ and $\widetilde{\mu}^{-1}(\zeta)/G$ are equivalent as K\"{a}hler manifolds.
\end{proposition}
\begin{proof}
    We define $H_p\subset T_p\mu^{-1}(\zeta)$ and $H'_{(p,t)}\subset T_{(p,t)}\widetilde{\mu}^{-1}(\zeta)$ as in the proof of \autoref{2}, and obtain the same commutative diagrams. Then we obtain
    \[\begin{split}
       J^\zeta\Bigl(X,f\displaystyle\frac{\partial}{\partial r}\Bigr)&=d\widetilde{\pi}J\widetilde{\Bigl(X,f\displaystyle\frac{\partial}{\partial r}\Bigr)}=d\widetilde{\pi}J\Bigl(\widetilde{X},f\displaystyle\frac{\partial}{\partial t}\Bigr) \\
       &=d\widetilde{\pi}\Bigl(\varphi \widetilde{X}-f\xi,\eta(\widetilde{X})\displaystyle\frac{\partial}{\partial t}\Bigr)\\
       &=\Bigl(d\pi(\varphi \widetilde{X}-f\xi),\eta(\widetilde{X})\displaystyle\frac{\partial}{\partial r}\Bigr)\\
       &=\Bigl(\varphi^\zeta X-f\xi^\zeta,\eta^\zeta(X)\displaystyle\frac{\partial}{\partial r}\Bigr).
   \end{split}\]

    We can also see that $h^\zeta=g^\zeta+dr^2$, therefore the K\"{a}hler structure $(h^\zeta,J^\zeta)$ coincides with one obtained by the cone construction $C({\mu}^{-1}(\zeta)/G)=({\mu}^{-1}(\zeta)/G)\times\mathbb{R}$.
\end{proof}

\vspace{0.1in}
Next we see the relationship between hyperK\"{a}hler reduction and 3-cosymplectic reduction. Let $(M,h,J_1,J_2,J_3)$ be a hyperK\"{a}hler manifold. Then $C^3(M):=M\times\mathbb{R}^3$ admits a natural 3-cosymplectic structure $(g,(\varphi_i,\xi_i,\eta_i)_{i=1,2,3})$ defined by 
\[g=h+\displaystyle\sum_{i=1}^3dt_i^2,\qquad\xi_i=\displaystyle\frac{\partial}{\partial t_i},\qquad\eta_i=dt_i,\]
\[\varphi_1\Bigl(X,f_1\displaystyle\frac{\partial}{\partial t_1},f_2\displaystyle\frac{\partial}{\partial t_2},f_3\displaystyle\frac{\partial}{\partial t_3}\Bigr)=\Bigl(J_1X,0,-f_3\displaystyle\frac{\partial}{\partial t_2},f_2\displaystyle\frac{\partial}{\partial t_3}\Bigr),\]
\[\varphi_2\Bigl(X,f_1\displaystyle\frac{\partial}{\partial t_1},f_2\displaystyle\frac{\partial}{\partial t_2},f_3\displaystyle\frac{\partial}{\partial t_3}\Bigr)=\Bigl(J_2X,f_3\displaystyle\frac{\partial}{\partial t_1},0,-f_1\displaystyle\frac{\partial}{\partial t_3}\Bigr),\]
\[\varphi_3\Bigl(X,f_1\displaystyle\frac{\partial}{\partial t_1},f_2\displaystyle\frac{\partial}{\partial t_2},f_3\displaystyle\frac{\partial}{\partial t_3}\Bigr)=\Bigl(J_3X,-f_2\displaystyle\frac{\partial}{\partial t_1},f_1\displaystyle\frac{\partial}{\partial t_2},0\Bigr),\]
where $(t_1,t_2,t_3)$ is the coordinate of $\mathbb{R}^3$ and $f_i\in C^\infty(C^3(M))$.

\begin{example}[\cite{montano2013topology}]
    Let $f$ be a hyperK\"{a}hler isometry on $(M,h,J_1,J_2,J_3)$. Define an action of $\mathbb{Z}^3$ on $C^3(M)$ by 
    \[(k_1,k_2,k_3)\cdot(p,t_1,t_2,t_3):=(f^{k_1+k_2+k_3}(p),t_1+k_1,t_2+k_2,t_3+k_3).\]
    This action is free and properly discontinuous, so $C^3(M)/\mathbb{Z}^3$ is a smooth manifold. Then $C^3(M)/\mathbb{Z}^3$ inherits a 3-cosymplectic structure. This is a 3-cosymplectic quotient for a trivial moment map on $C^3(M)$.$\End$
\end{example}

Assume that there is an action of $G$ on a hyperK\"{a}hler manifold $(M,h,J_1,J_2,J_3)$ which is Hamiltonian with respect to three symplectic structures and preserves $h,J_1,J_2,J_3$. Let $\mu:M\to\mathfrak{g}^\ast\otimes{\rm Im}\mathbb{H}$ be a hyperK\"{a}hler moment map. We define an action of $G$ on $C^3(M)$ by $g\cdot(p,t_1,t_2,t_3):=(gp,t_1,t_2,t_3)$. Then this action preserves the 3-cosymplectic structure on $C^3(M)$, and $\widetilde{\mu}=\mu\circ\textrm{pr}_M$ is a 3-cosymplectic moment map.
\begin{proposition}
    When a hyperK\"{a}hler quotient $\mu^{-1}(\zeta)/G$ is defined for $\zeta=\zeta_1i+\zeta_2j+\zeta_3k\in\mathfrak{g}^\ast\otimes{\rm Im}\mathbb{H}$, a 3-cosymplectic quotient $\widetilde{\mu}^{-1}(\zeta)/G$ is also defined. Moreover, $C^3(\mu^{-1}(\zeta)/G)$ and $\widetilde{\mu}^{-1}(\zeta)/G$ are equivalent as 3-cosymplectic manifolds.
\end{proposition}
\begin{proof}

    We define $N:=\mu_2^{-1}(\zeta_2)\cap\mu_3^{-1}(\zeta_3)$. Then clearly $\widetilde{\mu_2}^{-1}(\zeta_2)\cap\widetilde{\mu_3}^{-1}(\zeta_3)$ is diffeomorphic to $C^3(N)$. $N$ is endowed with a K\"{a}hler structure $(h|_N,J_1|_N)$ (see \cite{hitchin1987hyperkahler}), and the coK\"{a}hler structure $(g|_{C^3(N)},\varphi_1|_{C^3(N)},\xi_1|_{C^3(N)},\eta_1|_{C^3(N)})$ on $C^3(N)$ is obtained by applying cone constructions to $(N,h|_N,J_1|_N)$ three times. Moreover, by \autoref{2} and \autoref{2.01},
    \[C(C(C(N)))^{\zeta_1}\simeq C(C(C(N))^{\zeta_1})\simeq C(C(C(N)^{\zeta_1}))\simeq C(C(C(N^{\zeta_1})))\]
    holds as coK\"{a}hler manifolds. Therefore the reduced coK\"{a}hler structure $(g^\zeta,\varphi_1^\zeta,\xi_1^\zeta,\eta_1^\zeta)$ on $(C^3(M))^\zeta$ coincides with one of three coK\"{a}hler structures on $C^3(M^\zeta)$ obtained by $(h^\zeta,J_1^\zeta)$. 
    
    Repeating the same argument, we can see that the reduced 3-cosymplectic structure $(g^\zeta,(\varphi_i^\zeta,\xi_i^\zeta,\eta_i^\zeta)_{i=1,2,3})$ coincides with one obtained by the cone construction
    \[C^3({\mu}^{-1}(\zeta)/G)=({\mu}^{-1}(\zeta)/G)\times\mathbb{R}^3.\]
\end{proof}

\vspace{0.1in}
Conversely, for a given 3-cosymplectic manifold $(M,g,(\varphi_i,\xi_i,\eta_i)_{i=1,2,3})$, we can define a hyperK\"{a}hler structure $(h,J_1,J_2,J_3)$ on the cone $C(M)$ by
\[h=g+dt^2,\qquad J_i\Bigl(X,f\displaystyle\frac{\partial}{\partial t}\Bigr)=\Bigl(\varphi_i X-f\xi_i,\eta_i(X)\displaystyle\frac{\partial}{\partial t}\Bigr).\]
Assume that there is an action of $G$ on a 3-cosymplectic manifold $(M,g,(\varphi_i,\xi_i,\eta_i)_{i=1,2,3})$ which is Hamiltonian with respect to three cosymplectic structures and preserves three coK\"{a}hler structures. Let $\mu:M\to\mathfrak{g}^\ast\otimes{\rm Im}\mathbb{H}$ be a 3-cosymplectic moment map. We define an action of $G$ on $C(M)$ in the same way as in \autoref{2}. Then the action preserves the hyperK\"{a}hler structure on $C(M)$ and $\widetilde{\mu}=\mu\circ\textrm{pr}_M$ is a hyperK\"{a}hler moment map. We obtain the following.
\begin{proposition}
    When a 3-cosymplectic quotient $\mu^{-1}(\zeta)/G$ is defined for $\zeta=\zeta_1i+\zeta_2j+\zeta_3k\in\mathfrak{g}^\ast\otimes{\rm Im}\mathbb{H}$, a hyperK\"{a}hler quotient $\widetilde{\mu}^{-1}(\zeta)/G$ is also defined. Moreover, $C(\mu^{-1}(\zeta)/G)$ and $\widetilde{\mu}^{-1}(\zeta)/G$ are equivalent as hyperK\"{a}hler manifolds.
\end{proposition}
\begin{proof}

    There is a coK\"{a}hler structure $(g|_N,\varphi_1|_N,\xi_1|_N,\eta_1|_N)$ on $N:=\mu_2^{-1}(\zeta_2)\cap\mu_3^{-1}(\zeta_3)$, and by \autoref{2.01}
    \[(C(M)^\zeta,h^\zeta,J_1^\zeta)\simeq C(N)^{\zeta_1}\simeq C(N^{\zeta_1})\simeq\Bigl(C(M^\zeta),g^\zeta+dr^2,\varphi_1^\zeta+\eta_1^\zeta\otimes\displaystyle\frac{\partial}{\partial r}-dr\otimes\xi^\zeta_1\Bigr)\]
    holds as K\"{a}hler manifolds. Repeating the same argument, we can see that the hyperK\"{a}hler structure $(h^\zeta,J_1^\zeta,J_2^\zeta,J_3^\zeta)$ coincides with one obtained by the cone construction $C({\mu}^{-1}(\zeta)/G)=({\mu}^{-1}(\zeta)/G)\times\mathbb{R}$.
\end{proof}

\section{CoK\"{a}hler reduction of mapping tori}\label{mapping}

In general, every coK\"{a}hler manifold is locally the Riemannian product of a K\"{a}hler manifold with the real line (see \cite{cappelletti2013survey}, for example). Therefore it is important to find coK\"{a}hler manifolds which are not the global product of a K\"{a}hler manifold with $\mathbb{R}$ or $S^1$. Such coK\"{a}hler manifolds which are compact are obtained by the mapping torus procedure. For a manifold $S$ and a diffeomorphism $f:S\to S$, we define the \textit{mapping torus} $S_f$ as follows:
\[S_f=(S\times[0,1])/\{(p,0)\sim(f(p),1)\mid p\in S\}.\]
Note that there is a fibration $S\hookrightarrow S_f\overset{\textrm{pr}}{\twoheadrightarrow}S^1$. In the case that $S$ is endowed with a K\"{a}hler structure $(G,J)$ and $f:S\to S$ is a Hermitian isometry, $S_f$ admits a coK\"{a}hler structure. In fact, $S_f$ is diffeomorphic to $C(S)/\mathbb{Z}$ in \autoref{1.1}. 

Let $\Omega$ be the symplectic form of the K\"{a}hler manifold $(S,h,J)$. We extend $J$ and pullback $h,\Omega$ to $S\times[0,1]$, and they descend to $S_f$ since $f^\ast h=h$ and $f_\ast J=Jf_\ast$. We denote them $\widetilde{J},\widetilde{h},\widetilde{\Omega}$. Then we can write the coK\"{a}hler structure $(g,\varphi,\eta,\xi)$ on $S_f$ as follows:
\[\varphi=\widetilde{J},\qquad\eta=\textrm{pr}^\ast d\theta,\qquad\xi_{[(p,t)]}=\displaystyle\left. \frac{d}{ds}\right|_{s=0}[(p,t+s)],\]
\[g(X,Y)=\widetilde{h}(X,Y)+\eta(X)\eta(Y),\]
where $\theta$ is the coordinate of $S^1$ and $[(p,t)]$ denotes the equivalence class of $(p,t)$ with respect to the quotient $C(S)/\mathbb{Z}$. The corresponding 2-form is given by $\omega:=\widetilde{\Omega}$. 

It is known that any closed coK\"{a}hler manifold is in fact a K\"{a}hler mapping torus:
\begin{theorem}[Li\ \cite{li2008topology}]\label{2.1}
    A closed manifold $M$ admits a coK\"{a}hler structure if and only if there exists a K\"{a}hler manifold $(S,h,J)$ and a Hermitian isometry $f$ of $(S,h,J)$ such that $M$ is diffeomorphic to $S_f$.\hfill{$\Box$}
\end{theorem}

Assume that there is a free and proper Hamiltonian action of a Lie group $G$ on a K\"{a}hler manifold $(S,h,J)$ preserving the K\"{a}hler structure. Moreover, we suppose that a Hermitian isometry $f:S\to S$ of $(S,h,J)$ is equivariant with respect to the action of $G$. Then we can define an action of $G$ on the mapping torus $S_f$ by $g\cdot[(p,t)]=[(gp,t)]$.

\begin{proposition}\label{3}
    Let $\mu:S\to\mathfrak{g}^\ast$ be a moment map of the Hamiltonian action of $G$ on $S$. Then the action of $G$ on $S_f$ is Hamiltonian if and only if $\mu(f(p))=\mu(p)$ holds for some $p\in S$.
\end{proposition}
\begin{proof}
    Let $\widetilde{\mu}:S_f\to\mathfrak{g}^\ast$ be a cosymplectic moment map. A vector field on $S_f$ is locally has the form $X+a\frac{\partial}{\partial t}$, where $t$ is the coordinate of $\mathbb{R}$ and $a\in C^\infty(S_f)$. Then we have 
    \[d\widetilde{\mu}^{A}\Bigl(X+a\frac{\partial}{\partial t}\Bigr)=\omega\Bigl(A^\ast,X+a\frac{\partial}{\partial t}\Bigr)=\Omega(A^\ast,X)=d\mu^A(X)\]
    for any $A\in\mathfrak{g}$. Hence the map $\widetilde{\mu}$ locally has the form
    \[\widetilde{\mu}[(p,t)]=\mu(p)+\zeta\]
    for some $\zeta\in\mathfrak{g}^\ast$ (since both $\mu$ and $\widetilde{\mu}$ are equivariant, $\zeta$ must be central). Since $\widetilde{\mu}$ is globally defined, $\mu(f(p))=\mu(p)$ holds for any $p\in S$. Conversely, if $\mu(f(p))=\mu(p)$ holds for any $p\in S$, we obtain a cosymplectic moment map $\widetilde{\mu}$ by $\widetilde{\mu}[(p,t)]=\mu(p)+\zeta$ for any central $\zeta$. However, we have
    \[\begin{split}
        d(\mu^A\circ f)(X)&=d\mu^A(f_\ast X)=\Omega(A^\ast,f_\ast X)\\
       &=\Omega(f_\ast(f_\ast^{-1}A^\ast),f_\ast X)\\
       &=\Omega(f_\ast^{-1}A^\ast,X)=\Omega(A^\ast,X)\\
       &=d\mu^A(X)
   \end{split}\]
   for any $A\in\mathfrak{g}$ and $X\in\mathfrak{X}(S)$, hence it is sufficient that $\mu(f(p))=\mu(p)$ holds for some $p\in S$. Two moment maps on $S$ differs only by a constant, thus the condition $\mu(f(p))=\mu(p)$ is independent of the choice of a moment map $\mu$.
\end{proof}

\begin{remark}
    In \cite{itoh1977fixed} it was proved that if a compact K\"{a}hler manifold $(S,h,J)$ has positive holomorphic sectional curvature, then any Hermitian isometry $f$ of $(S,h,J)$ has a fixed point. In this case any Hamiltonian action on $S$ such that $f$ is equivariant induces a Hamiltonian action on $S_f$.$\End$
\end{remark}

Suppose that an equivariant Hermitian isometry $f$ satisfies the condition in \autoref{3}. We define a cosymplectic moment map $\widetilde{\mu}:S_f\to\mathfrak{g}^\ast$ by $\widetilde{\mu}[(p,t)]=\mu(p)$. Note that the  action of $G$ on $S_f$ is free and proper, and preserves the coK\"{a}hler structure.

Let $\zeta\in\mathfrak{g}^\ast$ be a regular value of the moment map $\widetilde{\mu}:S_f\to\mathfrak{g}^\ast$. Let $(S_f)^\zeta$ be the coK\"{a}hler quotient. In the case that $S$ is compact, $(S_f)^\zeta$ is a closed coK\"{a}hler manifold, and thus it is a mapping torus of some K\"{a}hler manifold from \autoref{2.1}. In the following we observe that it can be obtained by the K\"{a}hler quotient $S^\zeta=\mu^{-1}(\zeta)/G$ for the same value $\zeta$. From the condition $\mu(f(p))=\mu(p)$, $f$ preserves $\mu^{-1}(\zeta)$ and it descends to a map $f^\zeta:S^\zeta\to S^\zeta$ since $f$ is equivariant.
\begin{lemma}
$f^\zeta$ is a Hermitian isometry of $(S^\zeta,h^\zeta,J^\zeta)$.
\end{lemma}
\begin{proof}
    $f^\zeta$ is a diffeomorphism since $f^{-1}$ also descends to $S^\zeta$ and is the inverse of $f^\zeta$. We orthogonally decompose $T_p\mu^{-1}(\zeta)$ as 
    \[T_p\mu^{-1}(\zeta)=H_p\oplus\mathfrak{g}_p.\]
    Then for any $v\in H_p$, we have
    \[d\mu^Af_\ast(v)=d(\mu^A f)(v)=d\mu^A(v)=0\]
    from the proof of \autoref{3}, and also obtain
    \[h(f_\ast(v),A^\ast)=h(v,(f^{-1})_\ast(A^\ast))=h(v,A^\ast)=0,\]
    and thus $f_\ast(v)\in H_{f(p)}$.

    Let $\pi:\mu^{-1}(\zeta)\to S^\zeta$ be the projection of reduction. Since $f^\zeta\pi=\pi f$ and $f_\ast(H_p)=H_{f(p)}$, we obtain 
    \[(f^\zeta)_\ast=(d\pi|_{H_{f(p)}})f_\ast(d\pi|_{H_p})^{-1},\]
    so $(f^\zeta)^\ast h^\zeta=h^\zeta$ holds from $f^\ast h=h$ and the definition of $h^\zeta$. Similarly we can check that $(f^\zeta)_\ast J^\zeta=J^\zeta(f^\zeta)_\ast$.
\end{proof}
\begin{theorem}
    $(S_f)^\zeta$ is equivariant to $(S^\zeta)_{f^\zeta}$ as coK\"{a}hler manifolds.
\end{theorem}
\begin{proof}
    There is a diffeomorphism $\Phi:(S_f)^\zeta\to(S^\zeta)_{f^\zeta}$ defined by
    $\Phi(\widetilde{\pi}[(p,t)])=[(\pi(p),t)]$, where $\widetilde{\pi}:\widetilde{\mu}^{-1}(\zeta)\to (S_f)^\zeta$ is the projection of reduction. Since $\widetilde{\pi}[(p,t)]=\widetilde{\pi}[(q,s)]$ is equivalent to $[(\pi(p),t)]=[(\pi(q),s)]$, the map $\Phi$ is well-defined and one-to-one. Then the diagram
\[\xymatrix{{\mu}^{-1}(\zeta)\ar[d]_-{{i_t}}\ar[r]^-{{\pi}}&\mu^{-1}(\zeta)/G\ar[d]^-{{i_t}}\\
\widetilde{\mu}^{-1}(\zeta)\ar[r]_-{\widetilde{\pi}}&(S_f)^\zeta
    }\]
commutes, where $i_t$ denotes natural inclusions to mapping tori $p\mapsto[(p,t)]$. We orthogonally decompose $T_p\mu^{-1}(\zeta)$ and $T_{[(p,t)]}\widetilde{\mu}^{-1}(\zeta)$ as 
\[T_p\mu^{-1}(\zeta)=H_p\oplus\mathfrak{g}_p,\]
\[T_{[(p,t)]}\widetilde{\mu}^{-1}(\zeta)=H'_{[(p,t)]}\oplus\mathfrak{g}_{[(p,t)]},\]
respectively. Then we can easily check that $(i_t)_\ast(\mathfrak{g}_p)=\mathfrak{g}_{[(p,t)]}$ and thus obtain the following commutative diagram.
\[\xymatrix{H_{p}\ar[d]_-{(i_t)_\ast}\ar[r]^-{d{\pi}}_-{\simeq}&T_{\pi(p)}S^\zeta\ar[d]^-{(i_t)_\ast}\\
H'_{[(p,t)]}\ar[r]_-{d\widetilde{\pi}}^-{\simeq}&T_{\widetilde{\pi}[(p,t)]}(S_f)^\zeta
}\]

From $\varphi=\widetilde{J}$ and the diagram above, for $X\in T_{\pi(p)}S^\zeta$
\[\begin{split}
    \varphi^\zeta((i_t)_\ast X)&=d\widetilde{\pi}(\varphi\widetilde{((i_t)_\ast X)})\\
    &=d\widetilde{\pi}(\varphi((i_t)_\ast\widetilde{X}))\\
    &=d\widetilde{\pi}((i_t)_\ast J(\widetilde{X}))\\
    &=(i_t)_\ast d\pi(J(\widetilde{X}))=(i_t)_\ast J^\zeta(X)
\end{split}\]
holds, and thus $\varphi^\zeta=\widetilde{J^\zeta}$. Similarly we have $\widetilde{h}(\widetilde{X},\widetilde{Y})=\widetilde{h^\zeta}(X,Y)$ for $X,Y\in T_{\widetilde{\pi}[(p,t)]}(S_f)^\zeta$, hence we obtain
\[\begin{split}
    g^\zeta(X,Y)&=g(\widetilde{X},\widetilde{Y})=\widetilde{h}(\widetilde{X},\widetilde{Y})+\eta(\widetilde{X})\eta(\widetilde{Y})\\
   &=\widetilde{h^\zeta}(X,Y)+\eta^\zeta(X)\eta^\zeta(Y).
\end{split}\]

Let $\textrm{pr}_1:S_f\to S^1$ and $\textrm{pr}_2:(S^\zeta)_{f^\zeta}\to S^1$ be natural projections. Then the diagram
\[\xymatrix{
\widetilde{\mu}^{-1}(\zeta)\ar[rd]_-{\textrm{pr}_1}\ar[rr]^-{\widetilde{\pi}}&&(S_f)^\zeta\ar[ld]^-{\textrm{pr}_2} \\
&S^1&
    }\]
commutes, and thus we have $\widetilde{\pi}^\ast\textrm{pr}_2^\ast d\theta=\textrm{pr}_1^\ast d\theta=\eta|_{\widetilde{\mu}^{-1}(\zeta)}$. Hence $\textrm{pr}_2^\ast d\theta=\eta^\zeta$ holds.

From the above, the induced 2-form $\omega^\zeta$ and the Reeb vector field $\xi^\zeta$ on $(S_f)^\zeta$ are the same as those on $(S^\zeta)_{f^\zeta}$, and thus the cok\"{a}hler structure on $(S_f)^\zeta$ coincides with that of $(S^\zeta)_{f^\zeta}$.
\end{proof}

\section{A Perspective from Dynamical Systems}\label{dynamical}

In this section, we interpret our coK\"{a}hler reduction theorem from the viewpoint of dynamical systems. First, we explain why cosymplectic manifolds describe time-dependent Hamiltonian systems.

Let $(M,\omega)$ be a $2n$-dimensional symplectic manifold and $H\in C^\infty(M\times\mathbb{R})$. Then we consider a cosymplectic manifold $(M\times\mathbb{R},\ \eta:= d t,\ \omega_H:=\omega+ d H\wedge d t)$, where $t$ is the coordinate of $\mathbb{R}$. Using Darboux coordinates $(p_i,q_i)$ of $(M,\omega)$, the Reeb vector field $\xi_H$ of $(M\times\mathbb{R},\eta,\omega_H)$ is written as
\[\xi_H=\displaystyle\sum_{i=1}^{n}\Bigl(\displaystyle\frac{\partial H}{\partial p_i}\displaystyle\frac{\partial}{\partial q_i}-\displaystyle\frac{\partial H}{\partial q_i}\displaystyle\frac{\partial}{\partial p_i}\Bigr)+\displaystyle\frac{\partial}{\partial t}\]
and its integral curves are controlled by an ODE system
\begin{equation}\label{eq:5}
    \dot{q_i}=\displaystyle\frac{\partial H}{\partial p_i},\qquad\dot{p_i}=-\displaystyle\frac{\partial H}{\partial q_i},\qquad\dot{t}=1.
\end{equation}
Therefore we can consider $t$ as a time-parameter which parameterize ordinary Hamiltonian systems. Any cosymplectic manifold is locally the product of a symplectic manifold with the real line, so cosymplectic manifolds provide a good framework for time-dependent Hamiltonian systems.

\begin{remark}
    The Reeb vector field $\xi_H$ of $(M\times\mathbb{R},\eta,\omega_H)$ coincides with the vector field $\frac{\partial}{\partial t}+X_H$, where $X_H$ is the Hamiltonian vector field with respect to the cosymplectic structure $(\eta,\omega)$. The vector field $\frac{\partial}{\partial t}+X_H$ is called the \textit{evolution vector field} of the time-dependent system.$\End$
\end{remark}

For any almost coK\"{a}hler manifold $(M,g,\varphi,\xi,\eta)$, the following is known.
\begin{lemma}\label{lem:5}
   Integral curves of $\xi$ are geodesics with respect to $g$.
\end{lemma}
\begin{proof}
Let $\nabla$ be the Levi-Civita connection with respect to $g$. Then for any $X\in\mathfrak{X}(M)$,\ $g(\nabla_X\xi,\xi)=0$ holds since $X(g(\xi,\xi))=g(\nabla_X\xi,\xi)+g(\xi,\nabla_X\xi)$. Hence
\[\begin{split}
    0&=2 d\eta(\xi,X)=\xi(\eta(X))-\eta([\xi,X])\\
    &=\xi(g(X,\xi))-g(\nabla_\xi X-\nabla_X\xi,\xi)\\
    &=\xi(g(X,\xi))-g(\nabla_\xi X,\xi)+g(\nabla_X\xi,\xi)\\
    &=\xi(g(X,\xi))-g(\nabla_\xi X,\xi)\\
    &=g(X,\nabla_\xi\xi)\\
\end{split}\]
holds, and thus we obtain $\nabla_\xi\xi=0$.
\end{proof}

From \autoref{lem:5} and \autoref{100}, we immediately obtain the following.
\begin{proposition}
    Suppose that the cosymplectic manifold $(M\times\mathbb{R},\eta,\omega_H)$ admits a coK\"{a}hler structure and a Hamiltonian action of a Lie group which preserves the coK\"{a}hler structure. Then the image of a solution of (\ref{eq:5}) by the projection of reduction is a geodesic with respect to the reduced metric on $M^\zeta$.\hfill{$\Box$}
\end{proposition}

\begin{example}[cf.\cite{albert1989theoreme}]
    The motion of a solid in $\mathbb{R}^3$ with a fixed point (its center of inertia) is described by a manifold $T^\ast SO(3)$ equipped with the canonical symplectic form on $T^\ast SO(3)$ and a Hamiltonian $H:T^\ast SO(3)\simeq SO(3)\times\mathfrak{so}(3)^\ast\to\mathbb{R}$ defined by
    \[H(A,\alpha)=\displaystyle\sum_{1\leq i\leq3}\frac{\alpha_i^2}{M_i},\]
    where $\alpha_1,\alpha_2,\alpha_3$ are coefficients with respect to a suitable basis of $\mathfrak{so}(3)^\ast$ and $M_1,M_2,M_3$ are coefficients of the ellipsoid of inertia of the solid.

    Then the associated time-dependent Hamiltonian system is given by a cosymplectic manifold $(T^\ast SO(3)\times\mathbb{R},\ \eta:= d t,\ \omega_H:=\omega_{T^\ast SO(3)}+ d H\wedge d t)$. We define an action of $ SO (3)$ on $T^\ast SO(3)\times\mathbb{R}$ by
    \[B\cdot(A,\alpha,t)=(BA,\alpha,t)\]
    where $A,B\in SO(3),\ \alpha\in\mathfrak{so}(3)^\ast,\ t\in\mathbb{R}$. This action is a Hamiltonian action whose moment map is given by $\mu(A,\alpha,t)=\textrm{Ad}^\ast_A\alpha$. Then the reduced cosymplectic manifold at any non-zero vector $\zeta\in\mathfrak{so}(3)^\ast$ is $(\mathbb{S}^2\times\mathbb{R},\  d t,\ \omega_{\mathbb{S}^2}+ d H\wedge d t)$, where $\omega_{\mathbb{S}^2}$ is the standard symplectic form on $\mathbb{S}^2$ (note that the Hamiltonian $H$ is invariant by the action of $ SO(3)$ and thus descend to the quotient).
    
    Since $ SO(3)$ is compact, we can naturally construct a K\"{a}hler structure $(h,I)$ on $T^\ast SO(3)$ which is compatible with the canonical symplectic form $\omega_{T^\ast SO(3)}$. Then $T^\ast SO(3)\times\mathbb{R}$ is endowed with a $ SO(3)$-invariant almost coK\"{a}hler structure $(g,\varphi,\xi_H, d t)$ defined by 
    \[g(X,Y)=h(X,Y),\qquad g(X,\xi_H)=0,\qquad g(\xi_H,\xi_H)=1\]

    \[\varphi(X)=I(X),\qquad \varphi(\xi_H)=0\]
    for any $X,Y\in\ker d t=TM$. Since the Levi-Civita connection $\nabla$ of $g$ satisfies $\nabla\varphi=0$, $(g,\varphi,\xi,\eta)$ is a coK\"{a}hler structure, and the projection of time-dependent flows onto $\mathbb{S}^2\times\mathbb{R}$ are geodesics with respect to the reduced coK\"{a}hler metric.$\End$
\end{example}

\section{Conclusion}

In this paper, we proved reduction theorems for Hamiltonian actions on coK\"{a}hler manifolds and 3-cosymplectic manifolds, which are the polar opposites of Sasakian manifolds and 3-Sasakian manifolds, respectively. A notion of moment maps can also be defined for Lie group actions on contact manifolds \cite{albert1989theoreme,geiges1997constructions,willett2002contact}. However, the situation is quite different from that on symplectic manifolds, such as the moment map being uniquely determined by the action. On the other hand, Hamiltonian actions on cosymplectic manifolds have properties that are very similar to those on symplectic manifolds, and therefore, the results in this paper are natural odd-dimensional analogues of the reduction theorems by Hitchin {\it et al} \cite{hitchin1987hyperkahler}.

A future problem is to find further concrete examples of coK\"{a}hler reduction and 3-cosymplectic reduction. In addition, possible extensions of Albert's theory of Hamiltonian actions on cosymplectic manifolds, in directions different from those in this paper, include demonstrating a reduction theorem for actions of cosymplectic groupoids \cite{fernandes2023cosymplectic} and considering deformations of cosymplectic manifolds using moment maps \cite{nakamura2018deformations}.

\vspace{0.2in}
\bibliography{hoge}
\bibliographystyle{alpha}

\end{document}